\documentclass{article}

\usepackage{arxiv}

\usepackage[utf8]{inputenc} % allow utf-8 input
\usepackage[T1]{fontenc}    % use 8-bit T1 fonts
\usepackage{hyperref}       % hyperlinks
\usepackage{url}            % simple URL typesetting
\usepackage{booktabs}       % professional-quality tables
\usepackage{amsfonts}       % blackboard math symbols
\usepackage{nicefrac}       % compact symbols for 1/2, etc.
\usepackage{microtype}      % microtypography
\usepackage{lipsum}
\usepackage{tikz}
\usetikzlibrary{arrows,shapes,quotes}
\usetikzlibrary{patterns,decorations.pathreplacing}
\usepackage{amssymb, graphicx}
\usepackage{hyperref,lscape}
\usepackage{blkarray}
\usepackage{color}
\usepackage{amsmath}
\usepackage{mathtools}
\usepackage{mathrsfs}
\usepackage{chngcntr}
\usepackage{apptools}
\usepackage{blkarray}
\usepackage{csquotes}
\usepackage{bm}

\usepackage{parskip}
\usepackage{multicol}
\usepackage{amssymb}
\usepackage{cite}
\usepackage{amsthm}
\usepackage{makeidx}
\usepackage{graphicx}
\usepackage{mathrsfs,xcolor}
\usepackage{enumerate}
\usepackage{blkarray}
\usepackage{bm}
\usepackage{setspace}
\usepackage{algorithm,algorithmic,refcount}
\usepackage[left,pagewise, mathlines]{lineno}
\usepackage{tikz-cd} 
\usepackage{lineno}
\usepackage{comment}
\usepackage{longtable}
\usepackage{tabularx}

\numberwithin{table}{subsection}
\numberwithin{equation}{subsection}

\theoremstyle{plain}
\newtheorem{theorem}{Theorem}[subsection]

\newtheorem{proposition}{Proposition}[subsection]
\newtheorem{corollary}{Corollary}[subsection]

\theoremstyle{definition}
\newtheorem{definition}{Definition}[subsection]
\newtheorem{example}{Example}[subsection]

\title{Absolute Concentration Robustness in Rank-One Kinetic Systems}

\author{
Eduardo R.~Mendoza\thanks{Max Planck Institute of Biochemistry, Martinsried near Munich, Germany} \thanks{Faculty of Physics, Ludwig Maximilian University, Munich 80539 Germany} \\
  Mathematics and Statistics Department \\
  Center for Natural Sciences and Environmental Research \\
  De La Salle University \\
  Manila 0922, Philippines \\
  \texttt{eduardo.mendoza@dlsu.edu.ph} \\
  \And
  Dylan Antonio SJ.~Talabis \\
  Institute of Mathematical Sciences and Physics \\
  University of the Philippines \\
  Los Banos, Laguna 4031, Philippines \\
  \texttt{dstalabis1@up.edu.ph}
\And
  Editha C.~Jose \\
  Institute of Mathematical Sciences and Physics \\
  University of the Philippines \\ 
  Los Banos, Laguna 4031, Philippines \\
  \texttt{ecjose1@up.edu.ph} \\
   \And
  Lauro L.~Fontanil\thanks{Corresponding Author} \\
  Institute of Mathematical Sciences and Physics \\
  University of the Philippines \\ 
  Los Banos, Laguna 4031, Philippines \\
  \texttt{llfontanil@up.edu.ph} \\}

\begin{document}
\maketitle

\begin{abstract}
A kinetic system has an absolute concentration robustness (ACR) for a molecular species if its concentration remains the same in every positive steady state of the system. Just recently, a condition that sufficiently guarantees the existence of an ACR in a rank-one mass-action kinetic system was found. In this paper, it will be shown that this ACR criterion does not extend in general to power-law kinetic systems. Moreover, we also discussed in this paper a necessary condition for ACR in multistationary rank-one kinetic system which can be used in ACR analysis. Finally, a concept of equilibria variation for kinetic systems which are based on the number of the system's ACR species will be introduced here.   
\end{abstract}

% keywords can be removed
%\keywords{chemical reaction network theory \and power law kinetics \and complex balancing \and poly-PL kinetics}

\baselineskip=0.30in

\section{Introduction}

Robustness is the capacity of a system to maintain essential functions in the presence of internal or external stresses \cite{kitano}. It is an important characteristic that helps biological systems adapt to environmental changes and thrive. Several types of robustness have already been identified, but this paper focuses on the kind that requires the positive steady states of a system to possess certain qualities. 

As defined in \cite{shifein}, a biological system has an absolute concentration robustness (ACR) for a molecular species if in every positive steady state, the system admits, the concentration of the said species is the same. The identification of the steady states of a system is not an easy task, making this property difficult to determine. Shinar and Feinberg provided a sufficient condition that guarantees a deficiency-one mass-action kinetic (MAK) system exhibits an ACR \cite{shifein}. The said condition requires the corresponding network to have two nonterminal complexes that differ only in a specific species. This structural condition can be easily observed in most small networks and thus offers some advantages.

In 2018, Fortun et al. showed that the result of Shinar and Feinberg can be readily extended to deficiency-one power-law kinetic systems with reactant-determined interactions (PL-RDK) \cite{fort4}. A PL-RDK system is a kinetic system that is more general than the MAK system which requires reactions with the same reactant complex to have identical kinetic order vectors. In a MAK system, the elements of a kinetic order vector are the stoichiometric coefficients in the corresponding reactant complex.

The deficiency-one requirement of the results mentioned above is significantly limiting when it comes to analyzing the capacity of a system to admit an ACR. Fortunately, Fortun and Mendoza \cite{fort3} and Lao et al. \cite{lao} came up with more general results which do not require a network to have a specific deficiency.  Their results utilized the concept of network decomposition which is done by partitioning the reaction set of the network such that each partition generates a network (called subnetwork) smaller than the given network. The analysis then focuses on the low-deficiency subnetworks (with a deficiency of at most one) that contain a Shinar-Feinberg pair (SF-pair). An SF-pair is a pair of reactions with kinetic order vectors that only differ in a particular species. This result allows the ACR determination to be confined to the identified subnetworks that are relatively easier to handle. 

Recently, Meshkat et al. \cite{mesh} provided a necessary and sufficient condition for the existence of stable ACR in a rank-one MAK system. A system has stable ACR if it has an ACR such that each of its positive steady states is stable. However, this ACR criterion does not extend in general to PL-RDK systems, in contrast to the sufficient conditions for low deficiency systems by Horn and Jackson in 1972 \cite{horn} (as generalized by Fortun and Mendoza in 2021 \cite{fort3}) and Shinar and Feinberg in 2010 \cite{shifein}. Counterexamples and examples of PL-RDK systems for the ACR criterion are discussed in this paper. 

This paper also discussed a necessary condition for ACR in multistationary rank-one kinetic systems. The use of this condition in ACR analysis is illustrated by examples from classes of multistationary rank-one mass action systems introduced by Pantea and Voitiuk in 2022 \cite{voitiuk}. In addition, a concept of equilibria variation for kinetic systems based on the number of ACR species is introduced here. Its basic properties are derived and a general low bound is computed for deficiency zero PL-RDK systems. For multistationary rank-one systems, however, the necessary condition leads to a much sharper lower bound. 

%An analogous effort was conducted to see if this result can be readily extended to power-law systems just like the deficiency-dependent ACR extension in \cite{fort4, fort3}. As it turns out, this is not the case here since an example was found showing that the result of Meshkat et al. can not be directly adapted to power-law systems. On the other hand, subsets of PL-RDK systems for which the condition holds were identified. Thus, additional research is needed in order to determine which additional properties are required to establish the condition that works for rank-one power-law condition. 

This is how this paper was organized. Fundamental concepts on chemical reaction networks, kinetic systems, and robustness are provided in Section 2. In Section 3, the problem of extending the result of Meshkat et al. in PLK sytem is presented. A necessary condition for ACR in rank-one multistationary kinetic system is given in Section 4. ACR and equilibria variation are discussed in Section 5. A summary and an outlook are provided in the last section.

\section{Fundamentals of chemical reaction networks and kinetic systems}
\subsection{Structure of chemical reaction networks}

We review in this section some necessary concepts and results on chemical reaction network, the details of which can be found in \cite{fein2, arce, fort3}. 

First, we introduce some notations used in this paper. The sets of real number, nonnegative real numbers, and positive real numbers are denoted, respectively, by $\mathbb R, \mathbb R_{\geq 0}$, and $\mathbb R_{> 0}$. Given that $\mathscr I$ is a finite index set, $\mathbb R^{\mathscr I}$ denotes the usual vector space of real valued functions with domain $\mathscr I$ where addition, subtraction, and scalar multiplication are defined the usual way. For $x\in \mathbb R^{\mathscr I}$ and $i\in \mathscr I$, $x_i$ denotes the $i^{th}$ coordinate of $x$. Lastly, if $x\in \mathbb R^{\mathscr I}_{> 0}$ and $y\in \mathbb R^{\mathscr I}$, then $x^y\in \mathbb R^{\mathscr I}_{> 0}$ is defined to be $x^y = \displaystyle \prod_{i\in \mathscr I}x^{y^i}_i$. 

%Now, we give the definition of a chemical reaction network (CRN).

\begin{definition} A \textbf{chemical reaction network} is a triple $\mathscr N = (\mathscr S, \mathscr C, \mathscr R)$ of three non-empty finite sets:
\begin{enumerate}
    \item a set \textbf{species} $\mathscr S$;
    \item a set $\mathscr C$ of \textbf{complexes}, which are nonnegative integer linear combinations of the species; and
    \item a set $\mathscr R \subseteq \mathscr C \times \mathscr C$ of reactions such that
    \begin{itemize}
        \item $(y,y)\notin \mathscr R$ for all $y\in \mathscr C$, and
        \item for each $y\in \mathscr C$, there exists a $y'\in \mathscr C$ such that $(y,y')\in \mathscr R$ or $(y',y)\in \mathscr R$.
    \end{itemize}
\end{enumerate}
\end{definition}

\noindent The nonnegative coefficients of the species in a complex are referred to as \textbf{stoichiometric coefficients}.  In this paper, a reaction $R_i=(y_i,y'_i)$ is also denoted by $R_i: y_i\rightarrow y'_i$ and $y_i$ and $y'_i$ are called the \textbf{reactant} and \textbf{product complexes} of $R_i$, respectively. Further, we use the symbols $\sigma, \kappa$, and $\rho$ to denote the numbers of species, complexes, and reactions, respectively. The following example shows that a CRN can be represented by a digraph where the complexes and reactions serve as the digraph's vertices and arcs, respectively. %Note that the definition implies that the digraph representation of a CRN cannot have isolated vertices, loops, nor multiple arcs sharing common endpoints.  

\subsection*{Running example 1}

Chemical reaction networks (CRNs) can be represented as a directed graph. The vertices or nodes are the complexes and the reactions are the edges. The CRN is not unique and might not have a physical interpretation. Let us consider the following CRN: \\  

\begin{center}
\begin{tikzpicture}[baseline=(current  bounding  box.center)]
\tikzset{vertex/.style = {draw=none,fill=none}}
\tikzset{edge/.style = {bend left,->,> = latex', line width=0.20mm}} 
% vertices
\node[vertex] (1) at  (-4,1.5) {$X_1 + X_2 + X_3$};
\node[vertex] (2) at  (0,1.5) {$X_3$};
\node[vertex] (3) at  (-4,0) {$2X_1$};
\node[vertex] (4) at  (0,0) {$3X_1 + X_2$};
\node[vertex] (5) at  (4,0) {$4X_1 + 2X_2$};
\node[vertex] (6) at  (0,-1.5) {$4X_1+X_2$};
\node[vertex] (7) at  (4,-1.5) {$3X_1$};
\draw [edge]  (1) to["$k_1$"] (2);
\draw [edge]  (4) to["$k_2$"] (3);
\draw [edge]  (4) to["$k_3$"] (5);
\draw [edge]  (7) to["$k_4$"] (6);
\end{tikzpicture}
\end{center}

The $k_i$'s are called the reaction rate constants. We have $m=3$ (species), $n=7$ (complexes), $n_r=3$ (reactant complexes) and $r=4$ (reactions). Here, we can write
$$\mathscr{S}=\left\{ X_1, X_2, X_3\right\}, \quad \mathscr{C}=\left\{X_1 + X_2 + X_3, X_3, 2X_1, 3X_1 + X_2, 4X_1 + 2X_2, 4X_1+X_2, 3X_1 \right\}.$$

On the other hand, the set of reaction $\mathscr{R}$ consists of the following:
$$\begin{array}{l}
R_{1}: X_1 + X_2 + X_3 \rightarrow X_3 \\
R_{2}: 3X_1+X_2 \rightarrow 2X_1\\
\end{array} \quad \begin{array}{l}
R_{3}: 3X_1+X_2 \rightarrow 4X_1 + 2X_2 \\
R_{4}: 3X_1 \rightarrow 4X_1+X_2\\
\end{array}$$

We denote the CRN $\mathscr{N}$ as $\mathscr{N}= (\mathscr{S}, \mathscr{C}, \mathscr{R})$. The \textbf{linkage classes} of a CRN are the subgraphs of a reaction graph where for any complexes $C_i$, $C_j$ of the subgraph, there is a path between them. Thus, the number of linkage classes, denoted as $l$, of Running Example 1 is three ($l=3$). The linkage classes are:
$$\mathscr{L}_1=\left\{ R_1 \right\}, \quad \mathscr{L}_2=\left\{R_2,R_3 \right\} \quad \mathscr{L}_3=\left\{R_4 \right\}.$$
A subset of a linkage class where any two vertices are connected by a directed path in each direction is said to be a \textbf{strong linkage class}. Considering Running Example 1, there are no strong linkage classes whose number is denoted by $sl$. We also identify the \textbf{terminal strong linkage classes}, the number denoted as $t$, to be the strong linkage classes where there is no reaction from a complex in the strong linkage class to a complex outside the same strong linkage class.  The terminal strong linkage classes can be of two kinds: cycles (not necessarily simple) and singletons (which we call ``terminal points''). 

We now define important CRN classes.  A CRN is \textbf{weakly reversible} if every linkage class is a strong linkage class. A CRN is \textbf{t-minimal} if $t = l$, i.e. each linkage class has only one terminal strong linkage class. Let $n_r$ be the number of reactant complexes of a CRN. Then $n - n_r$ is the number of terminal points. A CRN is called \textbf{cycle-terminal} if and only if $n - n_r = 0$, i.e., each complex is a reactant complex. Clearly, the CRN of the Running Example 1 is both not t-minimal and weakly reversible. 

 With each reaction $y\rightarrow y'$, we associate a \textbf{reaction vector} obtained by subtracting the reactant complex $y$ from the product complex $y'$. The \textbf{stoichiometric subspace} $S$ of a CRN is the linear subspace of $\mathbb{R}^\mathscr{S}$ defined by
$$S := \text{span } \left\lbrace y' - y \in \mathbb{R}^\mathscr{S} \mid y\rightarrow y' \in \mathscr{R}\right\rbrace.$$

The \textbf{map of complexes} $\displaystyle{Y: \mathbb{R}^\mathscr{C} \rightarrow \mathbb{R}^\mathscr{S}_{\geq}}$ maps the basis vector $\omega_y$ to the complex $ y \in \mathscr{C}$. 
The \textbf{incidence map} $\displaystyle{I_a : \mathbb{R}^\mathscr{R} \rightarrow \mathbb{R}^\mathscr{C}}$ is defined by mapping for each reaction $\displaystyle{R_i: y \rightarrow y' \in \mathscr{R}}$, the basis vector $\omega_{R_i}$ (or simply $\omega_i$) to the vector $\omega_{y'}-\omega_{y} \in \mathscr{C}$. 
The \textbf{stoichiometric map} $\displaystyle{N: \mathbb{R}^\mathscr{R} \rightarrow \mathbb{R}^\mathscr{S}}$ is defined as $N = Y \circ  I_a$. 

In Running Example 1, the matrices $Y$ and $I_a$ are

$$Y=\begin{blockarray}{cccccccc}
C_1 & C_2 & C_3 & C_4 & C_5 & C_6 & C_7 \\
\begin{block}{[ccccccc]c}
1 & 0 & 2 & 3 & 4 & 4 & 3 & X_1 \\ 
1 & 0 & 0 & 1 & 2 & 1 & 0 & X_2 \\ 
1 & 1 & 0 & 0 & 0 & 0 & 0 & X_3 \\ 
\end{block}
\end{blockarray}$$

$$I_a=\begin{blockarray}{ccccc}
R_1 & R_2 & R_3 & R_4  \\
\begin{block}{[cccc]c}
-1 & 0 & 0 & 0 & C_1 \\
1 & 0 & 0 & 0 & C_2 \\
0 & 1 & 0 & 0 & C_3 \\
0 & -1 & -1 & 0 & C_4 \\
0 & 0 & 1 & 0 & C_5 \\
0 & 0 & 0 & 1 & C_6 \\
0 & 0 & 0 & -1 & C_7 \\
\end{block}
\end{blockarray}.$$

%We denote the product of $Y$ and $I_a$ as $N$ which is called the stoichiometric matrix.

The \textbf{deficiency} $\delta$ is defined as $\delta = n - l - \dim S$. This non-negative integer is, as Shinar and Feinberg pointed out in \cite{shifein2}, essentially a measure of the linear dependency of the network's reactions. In Running Example 1, the deficiency of the network is 3. 

\subsection{Dynamics of chemical reaction networks}

%The study of a CRN will not be complete without discussing the accompanying dynamics. The dynamics of a CRN requires a certain mapping that associate to each of its reactions some rate function. This mapping is called \textbf{kinetics}. 
A \textbf{kinetics} is an assignment of a rate function to each reaction in a CRN. A network $\mathscr N$ together with a kinetics $K$ is called a \textbf{chemical kinetic system} (CKS) and is denoted here by $(\mathscr N,K)$. %One of the most well-known kinetics is the \textbf{mass action kinetics} (MAK). In this paper, we focus more on kinetics that is more general than MAK, the \textbf{power law kinetics} (PLK). 
\textbf{Power law kinetics} (PLK) is identified by the \textbf{kinetic order matrix} which is an $\rho\times \sigma$ matrix $F=[F_{ij}]$, and vector $k\in \mathbb R^{\mathscr R}_{>0}$, called the \textbf{rate vector}.

\begin{definition} A kinetics $K: \mathbb R^{\mathscr S}_{>0} \rightarrow \mathbb R^{\mathscr R}$ is a \textbf{power law kinetics} if

\begin{center}
    $K_i(x)=k_ix^{F_{i,\cdot}}$ for $i=1,\dots, r$
\end{center}

\noindent where $k_i\in \mathbb R_{>0}$, $F_{i,j} \in \mathbb R$, and $F_{i,\cdot}$ is the row of $F$ associated to reaction $R_i$.
\end{definition}

We can classify a PLK system based on the kinetic orders assigned to its \textbf{branching reactions} (i.e., reactions sharing a common reactant complex). 

\begin{definition}
A PLK system has \textbf{reactant-determined kinetics} (of type PL-RDK) if for any two branching reactions $R_i, R_j\in \mathscr R$, the corresponding rows of kinetic orders in $F$ are identical, i.e., $F_{ih}=F_{jh}$ for $h=1, \dots,m$. Otherwise, a PLK system has \textbf{non-reactant-determined kinetics} (of type PL-NDK).
\end{definition}

Consider the CRN in Running example 1 with the following kinetic order matrix.

\begin{equation}
\nonumber
F=\begin{blockarray}{cccc}
X_1 & X_2 & X_3 \\
\begin{block}{[ccc]c}
0 & 0 & 2 & R_1 \\  
1 & 1 & 0 & R_2 \\  
1 & 1 & 0 & R_3 \\  
1 & 0 & 0 & R_4 \\ 
\end{block}
\end{blockarray}.
\end{equation}

\noindent Observe that $R_2$ and $R_3$ are two branching reactions whose corresponding rows in $F$ (or \textbf{kinetic order vectors}) are the same. Hence, the CRN is a PL-RDK system. 

The well-known \textbf{mass action kinetic system} (MAK) forms a subset of PL-RDK systems. In particular, MAK is given by $K_i(x)=k_ix^{Y_{.,j}}$ for all reactions $R_i: y_i \rightarrow y'_i \in \mathscr R$ with $k_i\in \mathbb R_{>0}$ (called \textbf{rate constant}). The vector $Y_{.,j}$ contains the stoichiometric coefficients of a reactant complex $y_i\in \mathscr C$.

%We now describe dynamics of a CKS through the following definition.

\begin{definition} The \textbf{species formation rate function} of a chemical kinetic system is the vector field

\begin{center}
    $f(c)=NK(c)=\displaystyle \sum_{y_i\rightarrow y'_i\in \mathscr R}K_i(c)(y'_i-y_i)$, where $c\in \mathbb R^{\mathscr S}_{\geq 0}$,
\end{center}
where $N$ is the $m \times r$ matrix, called \textbf{stoichiometric matrix},  whose columns are the reaction vectors of the system. \noindent The equation $dc/dt=f(c(t))$ is the \textbf{ODE or dynamical system} of the chemical kinetic system. An element $c^*$ of $\mathbb R^{\mathscr S}_{>0}$ such that $f(c^*)=0$ is called a \textbf{positive equilibrium} or \textbf{steady state} of the system. We use $E_+(\mathscr N,K)$ to denote the set of all positive equilibria of a CKS.
\end{definition}

Deficiency is one of the important parameters in CRNT to establish claims regarding the existence, multiplicity, finiteness and parametrization of the set of positive steady states. 

The dynamical system $f(x)$ (or species formation rate function (SFRF)) of the Running Example 1 can be written as
$$\left[ 
\begin{array}{c}
    \dot{X_1} \\
    \dot{X_2} \\
    \dot{X_3} \\
\end{array}
\right]=
\begin{blockarray}{cccc}
R_1 & R_2 & R_3 & R_4  \\
\begin{block}{[cccc]}
-1 & -1 & 1 & 1 \\
-1 & -1 & 1 & 1 \\
0 & 0 & 0 & 0 \\
\end{block}
\end{blockarray}
\left[ 
\begin{array}{c}
	k_1 X_3^2 \\
	k_2 X_1 X_2 \\
    k_3 X_1 X_2 \\
	k_4 X_1 \\
\end{array}
 \right]
=NK(x).$$

Analogous to the species formation rate function, we also have the complex formation rate function.

\begin{definition} The \textbf{complex formation rate function} $g: \mathbb R^{\mathscr S}_{>0}\rightarrow \mathbb R^{\mathscr C}$ of a chemical kinetic system is the vector field

\begin{center}
    $g(c)=I_aK(c)=\displaystyle \sum_{y_i\rightarrow y'_i\in \mathscr R}K_i(c)(\omega_{y'_i}-\omega_{y_i})$, where $c\in \mathbb R^{\mathscr S}_{\geq 0}$.
\end{center}

\noindent where $I_a$ is the incidence map. A CKS is \textbf{complex balanced} if it has complex balanced steady state, i.e., there is a composition $c^{**}\in \mathbb R_{>0}^{\mathscr S}$ such that $g(c^{**} )=0$. We denote by $Z_+(\mathscr N,K)$ the set of all complex balanced steady states of the system.
\end{definition}

%The following result on complex-balanced equilibria is vital in the discussion below. 

%\newtheorem{theorem}
%\begin{theorem}[Theorem 2B, Horn \cite{horn}] \label{1}  If a CKS has a complex balanced equilibrium then the underlying CRN is weakly reversible.
%\end{theorem}

%\newtheorem{theorem}
\begin{theorem}[Corollary 4.8, \cite{fein3}] \label{1.5}  If a CKS has deficiency 0, then its steady states are all complex balanced.
\end{theorem}

\subsection{A brief review of concentration robustness}
The concept of \textbf{absolute concentration robustness (ACR)} was first introduced by Shinar and Feinberg in their well-cited paper published in \textit{Science} \cite{shifein}. ACR pertains to a phenomenon in which a species in a kinetic system carries the same value for any positive steady state the network may admit regardless of initial conditions. In particular, a PL-RDK system $(\mathscr{N},K)$ has ACR in a species $X \in \mathscr{S}$ if there exists $c^*\in E_+(\mathscr{N},K)$ and for every other $c^{**} \in E_+(\mathscr{N},K)$,  we have $c^{**}_X =c^*_{X}$ where  $c^{**}_X$ and $c^*_{X}$ denote the concentrations of $X$ in $c^*$ and $c^{**}$, respectively.

Fortun and Mendoza \cite{fort3} emphasized that ACR as a dynamical property is conserved under dynamic equivalence, ie., they generate the same set of ordinary differential equations. They further investigated ACR in power law kinetic systems and derived novel results that guarantee ACR for some classes of PLK systems. For these PLK systems, the key property for ACR in a species $X$ is the presence of an SF-reaction pair. A pair of reactions in a PLK system is called a \textbf{Shinar-Feinberg pair} (or \textbf{SF-pair)} in a species $X$ if their kinetic order vectors differ only in $X$. A subnetwork of the PLK system is of \textbf{SF-type} if  it contains an SF-pair in $X$.

\section{The extension problem for the ACR criterion for rank-one systems}

Just recently, Meshkat et al. gave a necessary and sufficient condition that guarantees the existence of stable ACR in rank-one MAK systems \cite{mesh}. A system is said to have a stable ACR if it has an ACR where each of its positive steady states is stable. The condition requires the existence of an embedded one-species network that follows some structures described using arrow diagrams. 

\subsection{A review of the key results of Meshkat et al.}

A network $\mathscr N=\left(\mathscr S, \mathscr C, \mathscr R \right)$ is called one-species network if there is a species $A_i$ such that $(y,y')\in \mathscr R$ implies that $y_j= y_j'=0$ for all species $A_j\in S-\{A_i\}$. In other words, every complex in the network takes the form $kA_i$, where $k$ is a nonnegative integer. The following definition and example were taken from \cite{mesh}.

\begin{definition}
    Let $\mathscr N=\left(\{A\}, \mathscr C, \mathscr R \right)$  be a one-species network with $|\mathscr R|\geq 1$. Let every reaction of $\mathscr N$ be of the form $aA\rightarrow bA$, where $a,b\geq 0$ and $a\neq b$. Suppose $\mathscr N$ has $m$ distinct reactant complexes with $a_1<a_2<\dots<a_m$ as their stoichiometric species. The \textbf{arrow diagram} of $\mathscr N$, denoted by $\rho=(\rho_1,\rho_2,\dots,\rho_m)$, is the element of $\{\rightarrow,\leftarrow, \longleftrightarrow\}^m$ given by:
\begin{equation}\nonumber
   \rho_i = \left\{
\begin{array}{clrc@{\qquad}l}
\rightarrow & \textnormal{if for all reactions\;} a_iA \rightarrow bA \textnormal{\;in\;} G, \textnormal{\;it is the case that\;} b > a_i  \\
\leftarrow & \textnormal{if for all reactions\;} a_iA \rightarrow bA \textnormal{\;in\;} G, \textnormal{\;it is the case that\;} b < a_i\\
\longleftrightarrow & \textnormal{otherwise}
\end{array}
\right.
\end{equation}
\end{definition}

\begin{example}
Consider the network determined by $\{B\rightarrow A, 2A+B \rightarrow A+2B\}$. After removing species $A$, the embedded network has the following reaction set $\{0\leftarrow B \rightarrow 2B\}$ which has $(\longleftrightarrow)$ as arrow diagram. On the other hand, $\{0\rightarrow A, A\rightarrow 2A\}$ is the embedded network obtained after removing $B$ with arrow diagram $(\rightarrow,\leftarrow)$.
\end{example}

Here is the main result in \cite{mesh} that provides a necessary and sufficient condition that ensure the existence of ACR in some rank-one MAK systems. Their results are based on the idea of arrow diagrams.

\begin{theorem}\label{meshkatthm}
    Let $\mathscr N$ be a rank-one network with species $A_1, A_2, \dots, A_m$. Then, the following are equivalent:
    \begin{enumerate}
        \item $\mathscr N$ has stable ACR and admits a positive state.
        \item There is a species $A_{i^*}$ such that the following holds:
        \begin{enumerate}
            \item for the embedded network of $\mathscr N$ obtained by removing all species except $A_{i^*}$, the arrow diagram has one of the following forms:
           \begin{equation}
                \begin{array}{rl}               (\longleftrightarrow,\leftarrow,\leftarrow,\dots,\leftarrow), & (\rightarrow, \rightarrow, \dots, \rightarrow,\longleftrightarrow)  \\
                (\rightarrow, \rightarrow, \dots, \rightarrow,\longleftrightarrow, \leftarrow,\leftarrow,\dots,\leftarrow), &  (\rightarrow, \rightarrow, \dots, \rightarrow, \leftarrow,\leftarrow,\dots,\leftarrow)
              \end{array}
             \end{equation}
         \item the reactant complexes of $\mathscr N$ differ only in species $A_{i^*}$ (i.e., if $y$ and $\hat{y}$ are both reactant complexes, then $y_i=\hat{y_i}$ for all $i\neq i^*$).
      \end{enumerate}
    \end{enumerate}
\end{theorem}

Notice that in the above example, the embedded network obtained by removing species $A$ from $\mathscr N$ has $(\longleftrightarrow)$ as arrow diagram that falls under the enumerated arrow diagrams in the above theorem. Moreover, the reactant complexes in $\mathscr N$ differ only in $A$. It follows that $\mathscr N$ has a stable ACR.

\subsection{Examples of PL-RDK systems for which the ACR critera do not hold}

Here is the direct adaptation of the sufficient condition of Theorem \ref{meshkatthm} in PLK-systems and formalism. Let $(\mathscr N, K)$ be a rank-one PLK-system having species $A_1, A_2, \dots, A_m$. Also, let $F$ be the kinetic order matrix of the system. Then, $\mathscr N$ has an ACR if there is species $A_{j^*}$ such that the following holds

        \begin{enumerate}
            \item for the embedded network of $\mathscr N$ obtained by removing all species except $A_{j^*}$, the arrow diagram has one of the following forms:
           \begin{equation}
                \begin{array}{rl}               (\longleftrightarrow,\leftarrow,\leftarrow,\dots,\leftarrow), & (\rightarrow, \rightarrow, \dots, \rightarrow,\longleftrightarrow)  \\
                (\rightarrow, \rightarrow, \dots, \rightarrow,\longleftrightarrow, \leftarrow,\leftarrow,\dots,\leftarrow), &  (\rightarrow, \rightarrow, \dots, \rightarrow, \leftarrow,\leftarrow,\dots,\leftarrow)
              \end{array}
             \end{equation}
         \item the kinetic order vectors are pairwise SF-pairs in $A_{j^*}$ (i.e., $F_{ij}=F_{lj}$ for all $j\neq j^*$).
      \end{enumerate}

The following counterexample shows that the statement above is not necessarily true.

\begin{example}
Consider $\mathscr N=(\{A,B\},\mathscr C, \mathscr R)$, where $\mathscr R$ consists of the following reactions and rate constants:
\begin{equation}\nonumber
    \begin{array}{cl}
         R_1:& B {\xrightarrow[r_1]{}} A \\
         R_2:& 2A+B {\xrightarrow[r_2]{}} 3A\\
         R_3:& 3A+B {\xrightarrow[r_3]{}} 2A+2B\\
         R_4:& 4A+B {\xrightarrow[r_4]{}} 3A+2B\\
    \end{array}
\end{equation}
This is a rank-one network with $S=\textnormal{span\;}\{A-B\}$ as its stoichiometric subspace. Suppose that it is endowed with power law kinetics and has the following kinetic order matrix:

\begin{equation}\nonumber
F=\begin{blockarray}{ccc}
A & B \\
\begin{block}{[cc]c}
p_1 & q & R_1 \\  
p_2 & q & R_2 \\  
p_3 & q & R_3 \\  
p_4 & q & R_4\\ 
\end{block}
\end{blockarray}
\end{equation}
where $p_i$'s are integers. We obtained the following embedded network after removing $B$.

\begin{equation}\nonumber
\begin{array}{c}
    0\rightarrow A\\
    2A\leftrightarrow 3A\leftarrow 4A
\end{array}
\end{equation}
With the four distinct reactant complexes in this network, here is the corresponding arrow diagram $(\rightarrow, \rightarrow, \leftarrow, \leftarrow)$. Now, observe that $(\mathscr N, K)$ has the following ODEs:
\begin{equation}\nonumber
\begin{array}{cc}
    \dot{A}= & {r_1}A^{p_1}B^q + {r_2}A^{p_2}B^q - {r_3}A^{p_3}B^q - {r_4}A^{p_4}B^q \\
    \dot{B}= & -{r_1}A^{p_1}B^q - {r_2}A^{p_2}B^q + {r_3}A^{p_3}B^q + {r_4}A^{p_4}B^q 
\end{array}
\end{equation}
Solving for the positive equilibria,  one of these equations is equated to zero and get
\begin{equation}\nonumber
\begin{array}{cc}
    B^q({r_1}A^{p_1} + {r_2}A^{p_2} - {r_3}A^{p_3} - {r_4}A^{p_4} & = 0\\
    \Leftrightarrow {r_1}A^{p_1} + {r_2}A^{p_2} - {r_3}A^{p_3} - {r_4}A^{p_4} & = 0
\end{array}
\end{equation}
Observe that when $p_1<p_2<p_3<p_4$, the polynomial
\begin{equation}\label{poly}
    {r_1}A^{p_1} + {r_2}A^{p_2} - {r_3}A^{p_3} - {r_4}A^{p_4}
\end{equation}
has exactly one positive root by the Descarte's Rule of Signs. This means that the system has ACR in $A$. On the other hand, suppose that 
\begin{center}
    $     \begin{bmatrix} p_1 \\ p_2 \\ p_3 \\ p_4 \end{bmatrix} =
       \begin{bmatrix} 0 \\ 3 \\ 1 \\ 2 \end{bmatrix} \;$ and 
    $     \begin{bmatrix} r_1 \\ r_2 \\ r_3 \\ r_4 \end{bmatrix} =
       \begin{bmatrix} 4 \\ 1 \\ 2 \\ 3 \end{bmatrix} \;$
\end{center}
Then, the polynomial in (\ref{poly}) becomes 
\begin{equation}
    4+A^3-2A-3A^2=4-2A-3A^2+A^3
\end{equation}
This polynomial has two positive roots namely, $1$ and $1 + \sqrt{5}$. This means that the system does not
have an ACR in $A$ in this case. 
\end{example}

It can be observed in the above counterexample the crucial role played by $p_i$'s. If these parameters are not carefully chosen, ACR
may not be observed in the system. The problem becomes more complicated when $p_i$'s are non
integers since the solvability of the polynomial in (\ref{poly}) may not be determined immediately. This negatively affects the possibility of extending the result of Meshkat et al. to PLK-systems where $p_i$'s are allowed to be real numbers.

\subsection{The stable ACR criterion for homogenous PL-quotients of mass action systems}

In this section, we introduce the set of homogeneous monomial quotients of mass action systems (PL-MMK) where the stable ACR criterion holds. To define PL-MMK, we recall some definition and result from \cite{ACB2022}.

\begin{definition}
\label{PFF}
Two kinetics $K, K'$ in $\mathscr{K}_\Omega (\mathscr{N})$ are \textbf{positive function factor equivalent} (PPF-equivalent) if for all $x\in \mathbb{R}^{\mathscr{S}}_{>}$ and every reaction $q$, $\dfrac{K_q(x)}{K'_q(x)}$ is a positive function of $x$ only, i.e. independent of $q.$
\end{definition}

A key property of PFF-equivalence is expressed in the following proposition:

\begin{proposition}[\cite{ACB2022}]
If $K$, $K'$ are PPF, then $E_+(\mathscr{N}, K) =  E_+(\mathscr{N}, K')$.
\label{PFFresult}
\end{proposition}

Let $(\mathscr{N}, K)$ be a MAK system with $m$ species and $r$ reactions and $(\beta_1,\cdots,\beta_m)$ any real vector. Let $F_{MAK}$ and $F_\beta$ be the $r \times m$ matrices defined as follows:
\begin{itemize}
    \item $F_{MAK,q} = Y^\top_{\rho(q)}$, where $q$ is a reaction and $Y$ is the matrix of complexes of $(\mathscr{N}, K)$.
    \item $F_\beta$ with identical rows $\beta = (\beta_1,\cdots,\beta_m)$.
\end{itemize}

\begin{definition}
A \textbf{homogeneous PL quotient} of a MAK system $(\mathscr{N}, K)$ is a PLK system $(\mathscr{N}, K_{PLK})$ with the same rate constants and a kinetic order matrix $F_{PLK} = F_{MAK} - F_\beta$.
\end{definition}

We have the following result:

\begin{proposition}
The stable ACR criterion holds for homogeneous PL-quotients  of rank-one mass action systems.
\end{proposition}

\begin{proof}
Let $(\mathscr{N}, K_{PLK})$ be a homogeneous PL quotient of a rank-one MAK system $(\mathscr{N}, K)$ where $(\mathscr{N}, K)$  satisfies the stable ACR criterion. For each species, write  $X_i^{\alpha_i} = X_i^{\beta_i}(X_i^{\alpha_i - \beta_i})$. For a reaction $q$, each $K_q(x) = k_q \prod X_i^{\beta_i} \prod (X_i^{\alpha_i - \beta_i})$. Note that $\prod X_i^{\beta_i}$ is independent of the reaction $q$ and hence, from Definition \ref{PFF}, $K_{PLK}$ and $K$ are PFF-equivalent. From Proposition \ref{PFFresult},  $E_+(\mathscr{N}, K) =  E_+(\mathscr{N}, K_{PLK})$. Thus, their ACR coincide and, hence the stable ACR criterion also holds for $(\mathscr{N}, K_{PLK})$.
\end{proof}

\section{A necessary condition for ACR in multistationary rank-one kinetic systems}

In this section, we present a necessary condition for the occurrence of ACR species in any multistationary rank-one kinetic system. We introduce the class of co-conservative kinetic systems and show the multistationary rank-one systems of this kind do not have any ACR species.  We further illustrate the condition for a number of multistationary rank-one mass action systems, using a recently presented classification of such systems by Pantea and Voitiuk \cite{voitiuk}. 

\subsection{The necessary condition for ACR in rank-one stationary kinetic systems}

The necessary condition is the following fundamental observation:

\begin{theorem}
\label{4.1.1}
    If $\mathscr N$ has rank one and $\left (\mathscr N, K\right )$ is multistationary, then (the line) $S$ lies in the species hyperplane of every ACR species $X$. In other words, for any basis vector $v$ of $S$, its $X$-coordinate is $0$.
\end{theorem}

\begin{proof}
Since $\left (\mathscr N, K\right )$ is multistationary, there is a stoichiometric class that contains two distinct positive equilibria $x_1$ and $x_2$. In other words, $x_1 - x_2 \in S$ and $x_1 - x_2 \neq 0$. For any ACR species $X$ of $\left (\mathscr N, K\right )$, it follows that the $X$-th coordinate of $x_1 - x_2 $ is $0$. Since $x_1 - x_2 \neq 0$, it is a basis vector for the one-dimensional subspace $S$ of the rank-one system, proving the claim.
\end{proof}

The necessary condition immediately leads to an upper bound for the number of ACR species in a multistationary rank-one system:

\begin{corollary}
\label{cor1}
Let $m_{\textnormal{ACR}}$ be the number of $\textnormal{ACR}$ species of a rank-one multistationary system $\left (\mathscr N, K\right )$. Then $m_{\textnormal{ACR}} \leq m - |\textnormal{supp\;}v|$, where $\textnormal{supp\;}v$ is the support of any basis vector of $S$.
\end{corollary}

 Note that the right-hand side of the inequality in Corollary \ref{cor1} is just the number of zeros in $v$. Now, recall that a network is called conservative if the orthogonal complement of $S$ contains a positive vector. We hence define the following term.

\begin{definition}
A network is called \textbf{co-conservative} if $S$ contains a positive vector. 
\end{definition}

In general, a positive vector in $S$ can be a linear combination of reaction vectors with $0$ or negative coordinates. For example, for $m=3$, the reactions $X_1 \rightarrow X_1 + X_2$ and $2X_1 \rightarrow 3X_1 + 2X_3$ have the reaction vectors $(0,1,-1)$ and $(1, 0, 2)$ whose sum is $(1,1,1)$. In rank-one networks however, a positive vector requires the occurrence of a positive reaction vector $y'- y$. This in turn  implies a reaction $y \rightarrow y'$ with the following characteristics:
\begin{itemize}
    \item all species occur in the product complex;
    \item each species has a higher stoichiometric coefficient in the product complex compared to the corresponding species in the reactant complex; and
    \item no enzymatic regulation on a reaction.
\end{itemize}

\begin{corollary}
    If a rank-one and co-conservative network is multistationary, then it has no $\textnormal{ACR}$ species.
\end{corollary}

\begin{proof}
    Any non-zero vector $v$ in $S$ has only positive or negative coordinates, i.e., $|\textnormal{supp\;}v|=m$. Hence, $m_{\textnormal{ACR}}=0$. 
\end{proof}

\subsection{Examples from the multistationary rank-one mass action systems}

Pantea and Voitiuk introduced a complete classification of multistationary rank-one mass action systems in \cite{voitiuk}. The following table provides an overview of the eight classes that they identified.

\begin{table}[H]
  \centering

  \begin{tabularx}{\linewidth}{|X|X|}
    \hline
    \textbf{Network} & \textbf{Definition} \\ \hline
    Class 1-alt$^c$: 1-alt complete networks & has a 1D projection containing both $(\leftarrow,\rightarrow)$ and $(\rightarrow,\leftarrow)$ patterns  \\
    \hline
    Class 2-alt: 2-alternating & has a 1D projection containing both $(\leftarrow,\rightarrow,\leftarrow)$ and $(\rightarrow,\leftarrow,\rightarrow)$ patterns   \\
    \hline
    Class $Z$: zigzag network & has a 2D projection containing a zigzag  \\
    \hline
    Class $S_1$: one-source networks & contains exactly one source complex and two reactions of opposite directions \\
    \hline
    Class $S^z_2$ : two-source zigzag networks & has two species, contains exactly two source complexes and has a zigzag of slope -1 \\ \hline
    Class $L$: line networks & has two species and at least three source complexes satisfying some properties* \\ \hline
    Class $S^{nz}_2$ : two-source non-zigzag networks & an essential network that contains exactly two source complexes and $N \in$ 1-alt$^c - Z$. \\    \hline
    Class $C$: corner networks & an essential, 1-alt complete network that contains at least three source complexes satisfying some properties* \\
    \hline
  \end{tabularx}

\caption{Classes of rank-one networks (for details concerning the concepts and symbols used here, the readers are referred to \cite{voitiuk}).}
\label{rankonenet}
\end{table}

Under mass-action, our running example is a rank-one zigzag network but not a line, corner and two-source zigzag network. We illustrate Theorem \ref{4.1.1} with two examples from different classes of multistationary rank-one systems identified in \cite{voitiuk}.

\begin{example}[Class 1-alt$^c$]
The rank-one network below was the focus of Example 4.2 in \cite{voitiuk}. It was shown in that paper that it has the capacity for multiple positive and non-degenerate equilibria. Moreover, the network has stoichiometric subspace generated only by the vector $v=(1,1,1,0,-1,-1)$.
\begin{equation}\nonumber
    \begin{array}{cl}
         R_1:& B+2C+2E \rightarrow A+2B+3C+E \\
         R_2:& 2A+2B+C+2D+E \rightarrow A+B+2D+2E\\
         R_3:& A+3C+D+2E \rightarrow 2A+B+4C+D+E\\
         R_4:& 3A+3B+C+E \rightarrow 2A+2B+2E\\
    \end{array}
\end{equation}
\end{example}

\begin{example}[Class 2-alt]
The rank-one network below is a 2-alternating network with stoichiometric subspace generated  by the vector $v=(2,1,0)$. From Corollary 4.3 in \cite{voitiuk}, we conclude that it has the capacity for multiple positive and non-degenerate equilibria.
\begin{equation}\nonumber
    \begin{array}{cl}
         R_1:& 2X + 2Y \rightarrow Y \\
         R_2:& 3X + Y \rightarrow 5X + 2Y\\
         R_3:& 4X + 2Y + Z \rightarrow Z \\
         R_4:& 4X + 2Y + 2Z \rightarrow 2X + Y + 2Z\\
    \end{array}
\end{equation}
\end{example}

%We illustrate the main result in this subsection through a kinetic system given in the paper of Pantea and Voitiuk \cite{voitiuk}. The rank-one network below was the focus of Example 4.2 in \cite{voitiuk}. It was shown in that paper that it has the capacity for multiple positive and non-degenerate equilibria.

% The system' ODEs are as follows:

% \begin{equation} \nonumber
% \begin{array}{cl}
%     \dot{A}= & {r_1}BC^2E^2-{r_2}A^2B^2CD^2+{r_3}AC^3DE^2-{r_4}A^3B^3CE \\
%     \dot{B}= & {r_1}BC^2E^2-{r_2}A^2B^2CD^2+{r_3}AC^3DE^2-{r_4}A^3B^3CE \\ 
%     \dot{C}= & {r_1}BC^2E^2-{r_2}A^2B^2CD^2+{r_3}AC^3DE^2-{r_4}A^3B^3CE \\ 
%     \dot{D}= & 0 \\ 
%     \dot{E}= & -{r_1}BC^2E^2+{r_2}A^2B^2CD^2-{r_3}AC^3DE^2+{r_4}A^3B^3CE 
% \end{array}
% \end{equation}

Direct computations show that the preceding systems both have ACR in species $D$ and $Z$, respectively.  That is, their basis vectors of $S$ have $0$ coordinate in species $D$ and $Z$, respectively. This illustrates Theorem \ref{4.1.1}. 

%In \cite{voitiuk}, Pantea and Voitiuk classified rank-one mass action systems into 8 different classes based on reaction structure as summarized th following table:

%\includegraphics[]{pic_2.png}

%{\textcolor{red}{TO BE DONE in the next version-We are currently checking which of the general classes identified above satisfy the requirements of Theorem \ref{4.1.1}.}} 

\section{ACR and equilibria variation in kinetic systems}

In this section, we briefly discuss the relationship between concentration robustness and equilibria variability in general by introducing the concept of (positive) equilibria variation. We consider several examples to illustrate its use. We assume throughout the section that the kinetic system $(\mathscr N, K)$ is positively equilibrated, i.e., $E_+(\mathscr N, K) \neq \varnothing$.

\subsection{The equilibria variation of a kinetic system}

The motivation for the following definition comes from the observation that a kinetic system has a unique (positive) equilibrium in species space if and only if it possesses absolute concentration robustness in each of its species. 

\begin{definition}
The \textbf{(positive) equilibria variation} of a kinetic system $(\mathscr N, K)$ is the number of non-ACR species divided by the number of species, i.e., 
\begin{equation*}
    v_+(\mathscr N, K) = \dfrac{m-m_{ACR}}{m}.
\end{equation*}
\end{definition}

Clearly, the variation values lie between $0$ and $1$, and the following proposition characterizes the attainment of the extreme values:

\begin{proposition} Let $E_+(\mathscr N, K)$ be a kinetic system. Then,

\begin{enumerate}[i.]
    \item $v_+(\mathscr N, K) = 0 \Leftrightarrow |E_+(\mathscr N, K)| = 1$ and
    \item $v_+(\mathscr N, K) = 1  \Leftrightarrow (\mathscr N, K)$ has no ACR species
\end{enumerate}
\end{proposition}

The proofs follow directly from the definition. We have the following corollary for any multistationary system:

\begin{corollary} 
Let $(\mathscr N, K)$ be a kinetic system.

\begin{enumerate} [a.]
    \item If $(\mathscr N, K)$ is multistationary, then $v_+(\mathscr N, K) \geq \dfrac{1}{m}$. In particular, if $v_+(\mathscr N, K) = 0$, then $(\mathscr N, K)$ is monostationary.
    \item If $\mathscr N$ is open and  $v_+(\mathscr N, K) > 0$, then $(\mathscr N, K)$ is multistationary.
\end{enumerate}
\end{corollary}

\begin{proof}
    For $(a)$: Multistationarity implies at least two distinct positive equilibria, hence $m_{ACR} \leq m - 1$, and the claims follow. For $(b)$: if the network is open, then there is only one stoichiometric class. Hence, multistationarity is equivalent to the occurrence of at least two distinct equilibria in species space.
\end{proof}

\begin{example}
    Schmitz's model of the earth's pre-industrial carbon cycle system was analyzed by Fortun et al. in \cite{fort3.5}. Below is its corresponding reaction network.

    \begin{equation}
    \label{schmitz} 
    \begin{tikzcd}
M_5 \arrow[dd, "R_1 "'] \arrow[rd, "R_2", shift left] &                                                                                                    & M_2 \arrow[ld, "R_5"', shift right] \arrow[rd, "R_{11}"', shift right] \arrow[dd, "R_9"'] &                                                                           \\
                                                      & M_1  \arrow[lu, "R_3", shift left] \arrow[ru, "R_6"', shift right] \arrow[rd, "R_8"', shift right] &                                                                                           & M_4 \arrow[lu, "R_{10}"', shift right] \arrow[ld, "R_{12}"', shift right] \\
M_6 \arrow[ru, "R_4"']                                &                                                                                                    & M_3 \arrow[lu, "R_7"', shift right] \arrow[ru, "R_{13}"', shift right]                    &                                                                         
\end{tikzcd}
\end{equation}

\noindent In this network, $M_1, M_2, M_3, M_4, M_5,$ and $M_6$ stand for atmosphere, warm ocean surface waters, cool ocean surface waters, deep ocean waters, terrestrial biota, and soil and detritus, respectively. Important numbers of the network as well as its kinetic order matrix are given in the following.

\begin{figure}
\centering
\includegraphics[width=0.9\columnwidth]{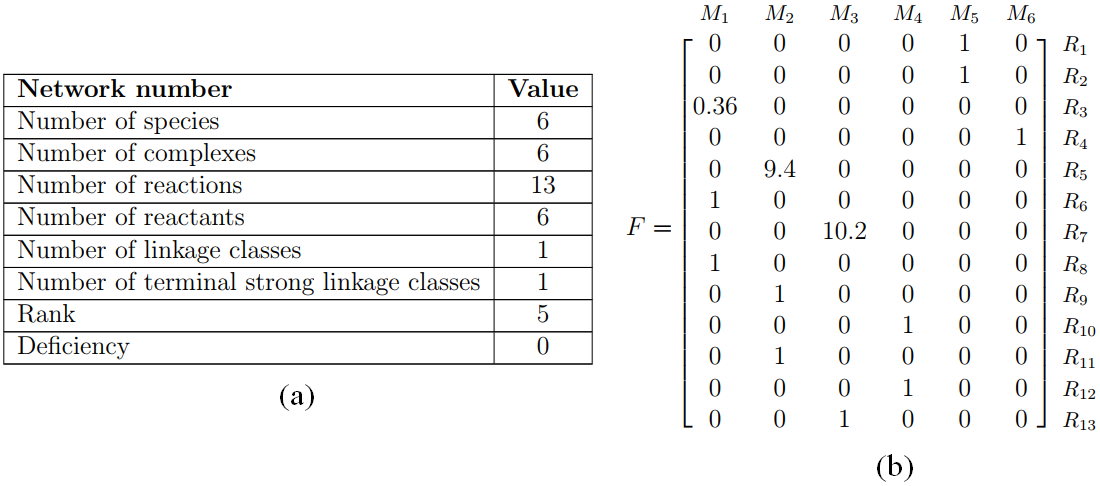}
\caption{(a) Some numbers related to the reaction network of Schmitz's model of the earth's pre-industrial carbon cycle system; (b) kinetic order matrix of the network.} 
\label{schmitz1} 
\end{figure}

\pagebreak

Fortun and Mendoza \cite{fort3.5} showed that the system has no ACR species, i.e., $v_+(\mathscr N, K) = 1 > \dfrac{1}{6}$. Since the system has a conservative, concordant, and weakly reversible network and weakly monotonic kinetics, by Theorem 6.6 of \cite{shinar}, it has a unique positive equilibrium in each stoichiometric class and consequently monostationary. Note also that the network has rank $5 < 6$, hence it is a closed network.
\end{example}

If the multistationary kinetic system has rank one, Theorem \ref{4.1.1}  provides a better lower bound for the equilibria variation.

\begin{proposition}
    Let $(\mathscr N, K)$ be a multistationary rank-one kinetic system and $v$ a basis vector of $S$. Then,  $v_+(\mathscr N, K) \geq \dfrac{|supp(v)|}{m}$.
\end{proposition}

\begin{proof}
    According to Theorem \ref{4.1.1},  $m_{ACR} \leq m - |supp(v)| \Leftrightarrow |supp(v)| \leq m - m_{ACR}$, leading to the new lower bound. Note that since $v$ is a basis vector, $|supp(v)|\geq 1$, confirming the improvement.
\end{proof}

One can derive a lower bound for equilibria variation for any kinetic system using the general species hyperplane criterion for ACR introduced by Hernandez and Mendoza in \cite{hernandez}. It is based on the following considerations:

%Let $(\mathscr N, K)$ be a kinetic system, with $\mathscr N=(\mathscr S, \mathscr C, \mathscr R)$ and $m$ is the number of species.

\begin{definition}
    For any species $X$, the $(m-1)$-dimensional subspace
    \begin{equation}
    \nonumber
        H_X:=\{x\in \mathbb R^{\mathscr S}| x_X=0\}
    \end{equation}
    is called the \textbf{species hyperplane} of $X$.
\end{definition}

For $U$ containing $\mathbb R_{>0}$, let $\phi: U\rightarrow \mathbb R$ be an injective map, i.e., $\phi: U\rightarrow \textnormal{Im}~\phi$ is a bijection. By component-wise application ($m$ times), we obtain a bijection $U^\mathscr S \rightarrow \mathbb R^{\mathscr S}$, which we also denote with $\phi$. We formulate our concepts for any subset $Y$ of $E_+(\mathscr N, K)$, although we are mainly interested in $Y=E_+(\mathscr N, K)$.

\begin{definition}
    For a subset $Y$ of $E_+(\mathscr N, K)$, the set 
    \begin{equation}
    \nonumber
    \Delta_{\phi}Y:=\{\phi(x)-\phi(x')| x,x'\in Y\}
    \end{equation}
    is called the \textbf{difference set of $\phi$-transformed equilibria} in $Y$, and its span $\langle     \Delta_{\phi}Y \rangle$ the difference space of $\phi$-transformed equilibria in $Y$.
\end{definition}
 
In Proposition 6 of \cite{hernandez}, it is shown that 
    \begin{equation}
    \nonumber
     m_{ACR} \leq m - \dim \langle \Delta_{\phi}E_+\rangle.
    \end{equation}		
 
\noindent It follows immediately that we have the following lower bound for the equilibria variation.

\begin{proposition}
    Let $(\mathscr N, K)$ be a kinetic system. Then, $\dfrac {\dim \langle \Delta_{\phi}E_+\rangle}{m}  \leq v_+(\mathscr N, K)$.
\end{proposition}

In general, it is challenging to compute $\dim \langle \Delta_{\phi}E_+\rangle$. However, for PLP systems, this can be done, and in the next section, we apply it to compute equilibria variation for weakly reversible, deficiency zero PL-RDK systems.

\subsection{Equilibria variation in deficiency zero PL-RDK systems}

A kinetic system $(\mathscr N, K)$ is a PLP (\textbf{positively equilibrated-log parametrized}) system if there is a reference equilibrium $x^*$ and a flux subspace $P_E$ of $\mathbb R^{\mathscr S}$ such that $E_+(\mathscr N, K) = \{ x \in \mathbb R^{\mathscr S}_{>} | \log x - \log x^* \in P_E^{\perp}\}$. For any PLP system, Lao et al. \cite{lao} showed that ACR is characterized by the species hyperplane criterion for PLP systems: $X$ is an ACR species if and only if it is contained in $H_X := \{x \in \mathbb R_{\mathscr S} | x_X = 0\}$. This implies that $m - \dim P_E \leq m - m_{ACR}$ and, hence, $1 - \dfrac{\dim P_E}{m} \leq v_+(\mathscr N, K)$.

%To apply this to deficiency zero PL-RDK systems, we recall the following concepts and results. 

Jose et al. \cite{ACB2022} showed that a CLP (\textbf{complex balanced-log parametrized}) system, i.e., $Z_+(\mathscr N, K) = \{x\in \mathbb R^{\mathscr S}_> |\log x - \log x^* \in P_Z^{\perp}\}$ is absolutely complex balanced, i.e., $E_+(\mathscr N, K)  = Z_+(\mathscr N, K)$ if and only if it is both CLP and PLP and $P_E = P_Z$. It was shown in \cite{muller} that any complex balanced PL-RDK system is a CLP system with $P_Z = \tilde{S}$. Since after Feinberg, any weakly reversible deficiency zero kinetic system is absolutely complex balanced, then a weakly reversible deficiency zero PL-RDK system is a PLP system with $P_E = \tilde{S}$. Hence, for any such system, $1 - \dfrac{\Tilde{s}}{m} \leq v_+(\mathscr N, K)$, where $\Tilde{s}$ is the kinetic rank of the system.

\begin{example}
    Fortun and collaborators studied variants of the Anderies et al. \cite{anderies} model of the earth's pre-industrial carbon cycle (see the reaction network below) in \cite{noel} and \cite{fort3.5}. 
   
\end{example}

\begin{equation}
\label{anderies} 
    \begin{tikzcd}
A_1+2A_2 \arrow[r, "R_1"]         & 2A_1+A_2                          \\
A_1+A_2 \arrow[r, "R_2"]          & 2A_2                              \\
A_2 \arrow[r, "R_3 ", shift left] & A_3  \arrow[l, "R_4", shift left]
\end{tikzcd}
\end{equation}

\noindent Note that $A_1$, $A_2$, and $A_3$ here stand for land, atmosphere, and ocean, respectively. Given below are some important numbers and the kinetic order matrix of the network in \ref{anderies}.

\begin{figure}[H]
\centering
\includegraphics[width=0.9\columnwidth]{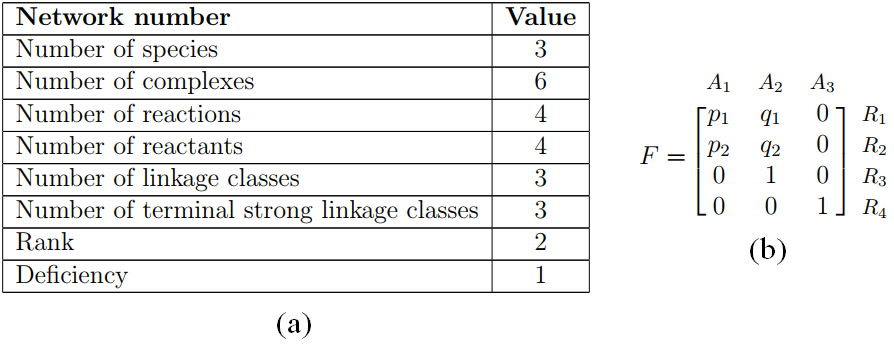}
\caption{(a) Some numbers related to the reaction network of the model of the earth's pre-industrial carbon cycle system given by Anderies et al. in \cite{anderies}; (b) kinetic order matrix of the network.} 
\label{schmitz2} 
\end{figure}

% \begin{center}
%     \centering
% \includegraphics[width=1.0\columnwidth]{pic_5.png}
% \end{center}

In \cite{fort3.5}, Fortun and Mendoza showed that the system is dynamically equivalent to a weakly reversible, deficiency zero system with $\Tilde{S}^{\perp} = \left \langle  \begin{pmatrix}
-1 & \dfrac{p_2-p_1}{q_2-q_1} & \dfrac{p_2-p_1}{q_2-q_1} 
\end{pmatrix}  \right \rangle$. Let $R$ denote the ratio $\dfrac{p_2 - p_1}{q_2 - q_1}$. Based on the value of $R$, one obtains three classes of Anderies system: $AND_> (R > 0)$ consists of multistationary systems with no ACR species, $AND_0 (R = 0)$ contains only monostationary systems with two ACR species and $AND_<$ contains both injective and non-injective systems but also no ACR species. The variants studied in [10] and [14] belong to $AND_>$ and $AND_0$, respectively.

Based on the previous results, we have $v_+(\mathscr N, K) = 1$ if $(\mathscr N, K)$ is contained in either $AND_>$ or $AND_<$, and $v_+(\mathscr N, K) = \dfrac{1}{3}$ if $(\mathscr N, K)$ is in $AND_0$.

\subsection{Equilibria variation of independent subnetworks}

In this section, we discuss equilibria variation for independent subnetworks. It is important in this regard to differentiate between the two concepts of subnetworks that we introduced in previous work: embedded and non-embedded. We hence begin with a review of these concepts and their properties relevant to absolute concentration robustness (ACR), which is the basis of our concept of equilibria variation. We then collect some relevant results from previous publications and add some new details. We conclude by using the results of the reaction network analysis of metabolic insulin signaling by Lubenia et al. \cite{lubenia} to illustrate the different concepts and relationships. We in fact show that the various inequalities between the different equilibria variation values are sharp, i.e., equality is achieved in various subnetworks of the insulin system. 

\subsubsection{Review of subnetwork properties}

In a subnetwork in a decomposition $\mathscr N=\mathscr N_1\cup \mathscr N_2\cup\cdots \cup\mathscr N_k$, often, a smaller number of species occurs than in the whole network. If $\mathscr N=(\mathscr S, \mathscr C, \mathscr R)$, we call a subnetwork $(\mathscr C', \mathscr R')$, where $\mathscr R'\subseteq \mathscr R$ and $\mathscr C'=\mathscr C|_{\mathscr R'}$, \textbf{embedded} if its species space is $\mathscr S$ and \textbf{non-embedded} if it has $\mathscr S|_{\mathscr C'}$ as species space, which we denote with $\mathscr S'$.

We use the embedded representation in a decomposition because it conveniently allows the set operations on equilibria sets. A basic fact is the following observation:

\begin{proposition}
    If $X$ is an ACR species of a subnetwork, then $X$ is an element of $\mathscr S'$.
\end{proposition}
   
This means that we need only one count of ACR species in a subnetwork, $m'_{ACR}$. For non-ACR species and equilibria variation, we have the following relationships:

\begin{proposition}
    Let $m'=|\mathscr S'|$ and $v_+(\mathscr N', K')$, $\Tilde{v}_+(\mathscr N', K')$ be the equilibria variations of the non-embedded and embedded subnetworks. Then,

    \begin{enumerate}[i.]
        \item $m-m'_{ACR}=(m'-m'_{ACR})+(m-m')$ and
        \item $0\leq \Tilde{v}_+(\mathscr N', K')-v_+(\mathscr N', K')\leq \dfrac{m-m'}{m}$
    \end{enumerate}    
\end{proposition}

\begin{proof}
$(i.)$ should be read as: the number of non-ACR species in an embedded network is the number of non-ACR species in the non-embedded network plus the number of non-occurring species, and $(ii.)$: dividing the equation in $(i.)$ by $m$ and using the inequality $\dfrac{m'-m'_{ACR}}{m}\leq \dfrac{m'-m'_{ACR}}{m'}$, we obtain $\Tilde{v}_+(\mathscr N', K')-v_+(\mathscr N', K')\leq \dfrac{m-m'}{m}$. On the other hand, the left-hand side, after forming a common denominator, is now equal to $\dfrac{m'_{ACR}(m-m')}{mm'}\geq 0$, since all factors are non-negative.
\end{proof}

\subsubsection{Equilibria variation of independent subnetworks}

In \cite{lao}, Proposition 4.4 states that if species $X$ has ACR in $\mathscr N_i$ and the decomposition is independent, then $X$ has ACR in $\mathscr N$, i.e., $\displaystyle |\mathscr S_{ACR, i}|\leq \left|  \bigcup_{i=1}^{k} \mathscr S_{ACR,i} \right| \leq |\mathscr S_{ACR}|$. We hence have the following corollary:

\begin{proposition}
    $v_+(\mathscr N, K)\leq \Tilde{v}_+(\mathscr N', K')$ for any independent subnetwork $\mathscr N'$.
\end{proposition}

\begin{proof}
    It follows that $m-m_{ACR}\leq m-m'_{ACR}$ and dividing both sides with $m$ results in the claim.
\end{proof} 

\subsubsection{The example of metabolic insulin signaling in healthy cells}

Lubenia et al. \cite{lubenia} constructed a mass action kinetic realization of the widely used model of metabolic insulin signaling (in healthy cells) by Sedaghat et al. \cite{sedaghat}. They used the kinetic system's finest independent decomposition (FID) to conduct an ACR analysis and showed that $m_{ACR} \geq 8$ for all rate constants (such that the system has positive equilibria) and $m_{ACR} = 8$ for some rate constants. Hernandez et al. [3] confirmed that  $m_{ACR} = 8$ for all rate constants with a positively equilibrated system. Hence, $v_+(\mathscr N, K) = \dfrac{20 - 8}{20} = \dfrac{12}{20} =\dfrac{3}{5} = 0.60$.

The table below (see Table \ref{lube}), which was taken from \cite{lubenia}, provides an overview of the characteristics of the FID subnetworks. Note that, with the exception of $\mathscr{N}_1$, all subnetworks are rank-one systems. Applying the Meshkat et al. criterion to these 9 subnetworks showed that only the one-species systems $\mathscr{N}_2$ and $\mathscr{N}_{10}$ had ACR for the species $X_6$ and $X_{20}$,  respectively. All other subnetworks had no ACR species. Hence, we have:

\begin{itemize}
    \item For $i=3,\dots,9,  v_+(\mathscr N, K) = 0.60 < 1 = v_+(\mathscr N_i, K_i) = \Tilde{v}_+(\mathscr N_i, K_i)$
    \item For $i=2,\dots,10,  v_+(\mathscr N_i, K_i) = 0< 0.60 = v_+(\mathscr N, K) = \Tilde{v}_+(\mathscr N_i, K_i)=\dfrac{19}{20}=0.95$
\end{itemize}

\begin{table}[H]
\begin{tabular}{|l|c|c|c|c|c|c|c|c|c|c|}
\hline
\textbf{Network numbers}        & $\mathscr N_1$ & $\mathscr N_2$ & $\mathscr N_3$ & $\mathscr N_4$ & $\mathscr N_5$ & $\mathscr N_6$ & $\mathscr N_7$ & $\mathscr N_8$ & $\mathscr N_9$ & $\mathscr N_{10}$ \\ \hline
Species                         & 7             & 1             & 4             & 3             & 3             & 2             & 3             & 3             & 4             & 1                \\ \hline
Complexes                       & 7             & 2             & 6             & 2             & 4             & 2             & 4             & 4             & 6             & 2                \\ \hline
Reactant complexes              & 7             & 2             & 3             & 2             & 2             & 2             & 2             & 2             & 4             & 2                \\ \hline
Reversible reactions            & 5             & 1             & 0             & 1             & 0             & 1             & 0             & 0             & 1             & 1                \\ \hline
Irreversible reactions          & 4             & 0             & 3             & 0             & 2             & 0             & 2             & 2             & 2             & 0                \\ \hline
Reactions                       & 14            & 2             & 3             & 2             & 2             & 2             & 2             & 2             & 4             & 2                \\ \hline
Linkage classes                 & 1             & 1             & 3             & 1             & 2             & 1             & 2             & 2             & 3             & 1                \\ \hline
Strong linkage classes          & 1             & 1             & 6             & 1             & 4             & 1             & 4             & 4             & 5             & 1                \\ \hline
Terminal strong linkage classes & 1             & 1             & 3             & 1             & 2             & 1             & 2             & 2             & 3             & 1                \\ \hline
Rank                            & 6             & 1             & 1             & 1             & 1             & 1             & 1             & 1             & 1             & 1                \\ \hline
Reactant rank                   & 7             & 1             & 3             & 2             & 2             & 2             & 2             & 2             & 4             & 1                \\ \hline
Deficiency                      & 0             & 0             & 2             & 0             & 1             & 0             & 1             & 1             & 2             & 0                \\ \hline
Reactant deficiency             & 0             & 1             & 0             & 0             & 0             & 0             & 0             & 0             & 0             & 1                \\ \hline
\end{tabular}
\caption{Some numbers of FID subnetworks.}
\label{lube}
\end{table}

%\begin{center}
%    \centering
%\includegraphics[width=1.0\columnwidth]%{pic3.png}
%\end{center}

The authors considered also the coarsening $\mathscr N_A \cup \mathscr N_B$, where $\mathscr N_A$ is the union of all deficiency zero subnetworks and $\mathscr N_B$ the union of all positive deficiency subnetworks.  They also showed that all ACR species are contained solely in $\mathscr N_A$.  Hence, in this case, we have: 

\begin{itemize}
    \item $v_+(\mathscr N_A, K_A) =\dfrac{13-8}{13}=0.38 < 0.6 = v_+(\mathscr N, K) = \Tilde{v}_+(\mathscr N_A, K_A)$
    \item $v_+(\mathscr N, K) = 0.6<1 = v_+(\mathscr N_B, K_B) = \Tilde{v}_+(\mathscr N_B, K_B)$
   
\end{itemize}

\noindent These examples show that all the inequalities in the propositions above are sharp.

\section{Summary and outlook}

This study was motivated by Meshkat et al.'s result which allows the analysis of rank-one MAK systems for possession of stable ACR. The result guarantees a MAK system such robustness if the two identified structural conditions are met. These conditions require the system to have an embedded network that follows certain structures and reactant complexes that differ in just one species. We attempted to extend this result to more general PLK systems but to no avail. We found an example illustrating how these conditions do not always work in a general setting. This means that, unlike other earlier results on ACR, the conditions in this result do not always insure the existence of ACR in a PLK system. 

On the other hand, we found a subclass of PLK systems where the stable ACR criterion of Meshkat et al. holds. We call this subclass homogenous monomial quotients of mass action systems or PL-MMK for short. This subclass is obtained by utilizing the set of rate constants of a given MAK system and its modified kinetic order matrix.

We also discovered a property that is necessarily present in a multistationary rank-one system that possesses an ACR.  Specifically, the corresponding result indicates that such system that has an ACR must have a basis vector generator of the stoichiometric subspace with 0 as a coordinate in the ACR species. We illustrated this result using a multistationary system that was given in the paper of Pantea and Voitiuk.

Finally, we considered the concept of 
 equilibria variation of independent subnetworks in this paper. We discussed some mathematical relationships of the equilibria variations of embedded and non-embedded subnetworks. These relationships were illustrated through the data shown in \cite{lubenia} that Lubenia et al. used for analyzing metabolic insulin signaling of healthy cells. It is important to note that ACR species of a subnetwork is always contained in the corresponding non-embedded subnetwork.  

 For future studies, one can look at the extension problem of the rank-one ACR criterion, that is, the identification of further kinetic systems (beyond mass action) for which it holds. One can also consider the exploration of relationships between multistationarity classes and the necessary condition for rank-one mass action systems as well as an extension of the Pantea-Voitiuk classification beyond mass action. Further, it is also interesting to identify further kinetics sets for which the general low bound can be computed or a much sharper alternative as in rank-one multistationary systems can be derived.

\vspace{0.5cm}
\noindent \textbf{Acknowledgement}

\noindent D. Talabis, E. Jose, and L. Fontanil acknowledge the support of the University of the Philippines Los Ba\~{n}os through the Basic Research Program.

\baselineskip=0.25in
%\bibliographystyle{abbrv}
%\bibliography{ref_acrdz}

\end{document}